\documentclass[reqno]{amsart}
\usepackage{graphicx}
\theoremstyle{theorem}
\newtheorem{theorem}{Theorem}[section]

\newtheorem{lemma}[theorem]{Lemma}
\newtheorem{prop}[theorem]{Proposition}
\newtheorem{definition}[theorem]{Definition}
\newtheorem{remark}[theorem]{Remark}
\newtheorem{condition}[theorem]{Condition}

\numberwithin{equation}{section}

\newcommand{\CC}{\mathbb{C}}
\newcommand{\RR}{\mathbb{R}}
\newcommand{\QQ}{\mathbb{Q}}
\newcommand{\ZZ}{\mathbb{Z}}

\newcommand{\CM}{\mathcal{M}}

\newcommand{\CL}{\mathcal{L}}

\newcommand{\RP}{\mathbb{R}P}
\newcommand{\CP}{\mathbb{C}P}

\newcommand{\AI}{A_\infty}

\newcommand{\NOV}{\Lambda_{nov}}
\begin{document}
\title[Non-displaceable Lagrangian submanifolds]
{Non-displaceable Lagrangian submanifolds and Floer cohomology with non-unitary line bundle}

\author{Cheol-Hyun Cho}
\address{Department of Mathematics, Seoul National University,
Kwanakgu Shinrim, San56-1 Seoul, South Korea, Email:
chocheol@snu.ac.kr}

\begin{abstract}
We show that in many examples the non-displaceability of Lagrangian submanifolds by Hamiltonian isotopy
can be proved via Lagrangian Floer cohomology with non-unitary line bundle.
The examples include all monotone Lagrangian
torus fibers in toric Fano manifold (which was also proven by Entov and Polterovich via the theory of symplectic quasi-states) and 
some non-monotone Lagrangian torus fibers.

We also extend the results by Oh and the author about the computations of Floer cohomology of Lagrangian torus fibers 
to the case of all toric Fano manifolds, removing the convexity assumption in the previous work. 
\end{abstract}

\maketitle

\bigskip
\footnote{This research has been partially supported by New Faculty Research Fund of Seoul National University.}
\section{Introduction}
This paper provides two improvements from the joint work with Yong-Geun Oh \cite{CO}. In \cite{CO}, we have
provided computations of Floer cohomology of Lagrangian torus fibers in toric Fano manifolds, which are
important examples of homological mirror symmetry, and symplectic topology of Fano manifolds \cite{HV},\cite{A},\cite{FOOO3},\cite{EP}. But in \cite{CO},
we have restricted ourselves to the case of some convex toric Fano manifolds because of the technical difficulty involving the transversality of moduli space. One of the result of this paper is to show how to remove this restriction, and prove
that the results of \cite{CO} hold for all toric Fano manifolds. 
The other result of this paper is to show that slight modification of Floer cohomology using non-unitary line bundle can be
successfully used to prove many Lagrangian intersection properties, which was not possible in the standard definition.
  
Lagrangian Floer cohomology was first defined by Floer \cite{Fl}, and generalized to the monotone case by
Yong-Geun Oh \cite{Oh1}. The definition in full generality, including obstruction  and deformation theory
has been established by the ingenious work of Fukaya, Oh, Ohta and Ono \cite{FOOO}.
The main feature of Floer cohomology of a Lagrangian submanifold is that 
Floer cohomology ring $HF(L,\psi(L))$ is independent of the choice of the 
Hamiltonian isotopy $\psi$. This provides a tool to study an intersection theory
of Lagrangian submanifolds, which has played an important role to study symplectic topology in
the last two decades.

We show examples of Lagrangian submanifolds which cannot be displaced from itself using Hamiltonian isotopy (non-displaceable, for short). The main tool is the Lagrangian Floer cohomology, with complex-valued two form 
on the symplectic manifold, which we call, the Floer cohomology with non-unitary line bundle. This
turns out to be more efficient than the standard definition for the intersection theoretic applications.

The new non-displaceable examples include  all monotone Lagrangian torus fibers, which was also proved by Entov and
Polterovich via a different method, namely the theory of symplectic quasi-states and quasi-measures \cite{EP2}.
We also find some new non-monotone non-displaceable Lagrangian fibers and general criterian to find one in Fano manifolds.  Recall that in the cases of the non-displaceable examples known before via Floer theoretic methods, like the Clifford torus (\cite{C},\cite{BEP}), or $\RP^n$ (\cite{Oh2}), 
the proof relies on the presence of strong symmetry (balanced) for the cancellation of quantum contributions
 so that Floer cohomology is isomorphic to the singular cohomology.
 In these new examples, they are not balanced but the new effect enables us more freedom so that we can still cancel quantum contributions out.
 
To explain this effect, recall that Lagrangian Floer cohomology has a generalization, which came from another motivation, the homological mirror symmetry conjecture \cite{K},\cite{Fu}. On this generalization, one allows, a flat unitary line bundle on a Lagrangian submanifold, or more
generally, one allows a unitary line bundle on $L$ whose curvature equals a closed two form $2\pi\sqrt{-1}B$ (defined  on $M$). This is usually called a
$B$-field. The modification we consider is a non-unitary version of them for short. As far as we know, non-unitary case has not been considered before in these cases. This maybe is due to the fact that the Floer homology with
unitary bundle has been used mainly for homological mirror symmetry and in such cases
 one usually identifies the formal parameter in the Novikov ring with a complex number. 
 For example in \cite{CO}, we have
identified $T^{2\pi}$ with a complex number $e^{-1}$ to show the equivalence of the obstruction $m_0$ with
the Landau-Ginzburg superpotential.

 But with such identification, 
the non-unitary version we consider in this paper is equivalent to the standard version (with a different symplectic 
form) in good cases. Also, there is a delicate point we also address in this paper that 
when formal parameter is identified with a complex number, the related Floer cohomology (which is called convergent version) does not provide
informations which are symplectically invariant.
(In the last section, we explain the example of Hirzbruch surface, where
some fibers have non-vanishing convergent Bott-Morse Floer cohomology, but still is displaceable by
some Hamiltonian isotopy. Hence, we show by this example that $HF(L,L) \neq HF(L,\psi(L))$ for the convergent version
of Floer cohomology.)

The computations in this paper, and that of \cite{CO},\cite{C} was carried out
using the Bott-Morse version of Lagrangian Floer homology, or an $m_1$-homology of
an $\AI$-algebra of Lagrangian submanifolds defined in \cite{FOOO}. It was shown in \cite{FOOO} that
transversal and Bott-Morse versions of Floer cohomologies are isomorphic. To compute the Bott-Morse version of Floer cohomology of a Lagrangian submanifold, one has to classify (pseudo)-holomorphic discs. In \cite{CO}, we classify all the holomorphic discs
with boundary on any Lagrangian torus fiber, and prove the Fredholm regularity of the standard complex structure for
all holomorphic discs. This enabled us to compute the Floer cohomology, except  one technical problem,
which involves the non-regularity of the standard complex structure for holomorphic spheres.
Hence, in \cite{CO}, we were only able to validate our results for convex toric Fano manifolds.

In this paper, we describe how to resolve this problem, using
a simple argument and  the machinery of \cite{FOOO} by combining the use of holomorphic discs and abstract perturbations
of Kuranishi structures together. Hence, we are able to extend the results of \cite{CO} to all toric Fano manifolds.

We would like to thank L. Polterovich for asking for the  Floer theoretic proof of his results with M. Entov.
We also thank for M. Entov, Kenji Fukaya and Yong-Geun Oh for helpful comments on the paper.
We remark that the idea in this paper \cite{C3}, has been generalized by Fukaya, Oh, Ohta and Ono \cite{FOOO3}.
\section{Floer cohomology with a closed two form $\widetilde{B}$}
\subsection{Floer cohomology with $B$-fields}
We first recall the definition of Floer cohomology with $B$-fields following the article of Fukaya \cite{Fu} (See also,\cite{FOOO}). We will consider a slight generalization of it in the next subsection.

Let $(M,\omega)$ be a symplectic manifold, and $B$ a closed real-valued two form on $M$, which is called $B$-field (in fact $B$ is from $H^2(M,\ZZ)$).
Let $L_i$ be a Lagrangian submanifold in $M$, and $\CL_i$ be a complex line bundle over $L_i$ for $i=1,2$.
Let $\nabla_i$ be a unitary connection on $\CL_i$, and it is required that
the curvature $F_{\nabla_i}$ of $\nabla_i$ satisfies for $i=1,2$, 
\begin{equation}\label{curv}
F_{\nabla_i} = 2 \pi \sqrt{-1} B.
\end{equation}

The definitions and properties of the standard Lagrangian Floer cohomology extends to the case with $B$-fields without much difficulty. We first recall the definition in the transversal case.

The $\NOV$-module $D_{(L_1,\CL_1),(L_2,\CL_2)}$ is defined as
$$D_{(L_1,\CL_1),(L_2,\CL_2)} = \oplus_{ p \in L_1 \cap L_2} Hom(\CL_{1,p},\CL_{2,p}) \otimes \Lambda_{nov,\CC}.$$
The differential $ \delta$ ($ =n_{0,0}$ in \cite{FOOO}) is defined using the moduli space of holomorphic strips.
For $p,q \in L_1 \cap L_2$, consider the following moduli space
$$\widetilde{M}(L_1,L_2;p,q) =\{ \phi:[0,1] \times \RR \to M| \textrm{Condition \ref{property}}. \}$$
\begin{condition}\label{property}
\begin{enumerate}
\item $\phi$ is pseudo-holomorphic.
\item $\phi(0,\tau) \in L_1, \phi(1,\tau) \in L_2.$
\item $\lim_{\tau \to -\infty} \phi(t,\tau) = p, \lim_{\tau \to \infty} \phi(t,\tau) = q$
\end{enumerate}
\end{condition}
After moding out by $\RR$-action, we get a quotient space $\CM(L_1,L_2;p,q)$. The moduli space
can be divided according to the homotopy classes, say $\beta$ of the maps $\phi$'s,
and denoted by $\CM(L_1,L_2;p,q;\beta)$. 
The holonomy contribution $Hol(\phi): Hom(\CL_{1,p},\CL_{2,p}) \to  Hom(\CL_{1,q},\CL_{2,q})$
is defined by 
$$Hol(\phi)(\alpha) = h_{\phi(\{ 1 \} \times \RR )}(\CL_2) \circ \alpha \circ h_{\phi(\{0\} \times \RR)}(\CL_1)^{-1}.$$
Here $h_{\phi(\{1\} \times \RR)}(\CL_2): \CL_{2,p} \to \CL_{2,q}$ is a parallel transport of $(\CL_2,\nabla_2)$
along the path $\phi(\{1\} \times \RR)$. The map $h_{\phi(\{0\} \times \RR)}(\CL_1)$ is defined similarly.

The symplectic area is given by $$\omega([\phi]) = \int_{D^2} \phi^*\omega \in \RR. $$

\begin{lemma}$($\cite{Fu} Lemma 2.7$)$
The following expression for $\phi$ only depends on the homotopy class of $\phi$.
$$exp(2\pi \sqrt{-1} \int_{D^2} \phi^*B) Hol(\phi(\alpha)) \otimes T^{\omega(\beta)}$$
\end{lemma}
Then the boundary map $\delta:D_{(L_1,\CL_1),(L_2,\CL_2)} \to D_{(L_1,\CL_1),(L_2,\CL_2)}$ is defined 
for $\alpha \in Hom(\CL_{1,p},\CL_{2,p})$ as
$$\delta(\alpha) = \sum_{q,\beta,\phi \; \textrm{with} \; \textrm{vir. dim.}(\CM(L_1,L_2;p,q;\beta)=0)}
\pm exp(2\pi \sqrt{-1} \int_{D^2} \phi^* B) \cdot Hol(\phi(\alpha)) \otimes T^{\omega(\beta)}. $$
Here signs can be determined explicitly if $L_1,L_2$ are (relatively) spin and are equipped with  choices of
spin structures.
Additional operations $n_{k,l}$ may be defined as in (\cite{FOOO}) to give $D$ the structure of $\AI$-bimodule over the $\AI$-algebras of $L_1$ and $L_2$. Consideration of $\AI$-bimodule structure
is essential when one actually needs to deform the $\AI$-algebra to define Floer cohomology, but in our
case, it is not nessesary, hence we refer readers to \cite{FOOO} for more details on them.
For the case of monotone Lagrangian submanifolds, it is not necessary to use the full machinery, and
we may use the approach of Oh \cite{Oh2}, Cho-Oh \cite{CO}, or Biran-Cornea \cite{BC} with the corresponding modifications. In the case considered in this paper, $\delta^2 =0$ holds when $(L_2,\CL_2,\nabla_2)$ is Hamiltonian equivalent to $(L_1,\CL_1,\nabla_1)$ in the following sense.
i.e. there exist a function $f:M \times [0,1] \to \RR$, a complex line bundle $\CL \to L \times [0,1]$ and
its connection $\nabla$ with the following properties. Let $\psi_t$ be the Hamiltonian isotopy generated
by $f$.
\begin{enumerate}
\item $\psi_1(L_1) = L_2$.
\item $(\CL, \nabla)|_{L \times \{0\}} = (\CL_1,\nabla_1),\;\; (\CL, \nabla)|_{L \times \{1\}} = \psi_1^*(\CL_2,\nabla_2).$

\item $F_\nabla = 2\pi \sqrt{-1}\, \Psi^*B.$ Here $\Psi:L_1 \times [0,1] \to M$ defined by $\Psi(x,t) = \psi_t(x)$.
\end{enumerate}

We remark that given $(L_1,\CL_1,\nabla_1)$ and a Hamiltonian isotopy $\psi_2$,  one can find $(\CL_3,\nabla_3)$  such that
$(L_1,\CL_1,\nabla_1)$ and $(\psi_2(L_1),\CL_3,\nabla_3)$ are Hamiltonian equivalent to each other in the above sense.
(For this, choose $\CL = \pi_1^* \CL_1$, $\nabla'= \pi_1^* \nabla_1$ using the projection $\pi_1:L_1 \times [0,1] \to L_1$ and notice that $F_{\nabla'} - 2\pi \sqrt{-1}\, \Psi^*B$ is exact.
Choose one form $\alpha$ on $L_1\times [0,1]$ with 
$F_{\nabla'} - 2\pi \sqrt{-1}\, \Psi^*B = d\alpha$. Take $\nabla = \nabla' - \alpha$, and define $(\CL_2,\nabla_2)$ using the above relation.)

The homology of $\delta$ is the Floer cohomology $HF((L_1,\CL_1,\nabla_1),(L_2,\CL_2,\nabla_2))$ for Hamiltonian equivalent pairs, and  it is independent of a Hamiltonian isotopy $\psi$, and it is isomorphic to the Bott-Morse Floer cohomology, which is a homology of an $\AI$-algebra of Lagrangian submanifold  with modification as follows.
(We refer readers to \cite{FOOO} for details and notations) 
The $\AI$-algebra operations $m_k:C(L)^{\otimes k} \to C(L)$ for  $k \in \mathbb{N} \cup \{0\}$ on geometric chains on $L$ are defined as
\begin{equation}
m_{k,\beta}(P_1,\cdots,P_k) = 
[\CM_1(L,J,\beta,P_1,P_2,\cdots,P_k),ev_0]  \cdot T^{\omega(\beta)},
\end{equation}
where $P_i$'s are chains on $L$, and $\CM(L,J,\beta,\cdots,P_*)$ is the moduli space of $J$-holomorphic discs with
homotopy class $\beta$ intersecting $P_1,\cdots,P_k$ with an additional marked point for evaluation.
And this may be modified with the presence of $B$-field as
\begin{equation}
m_{k,\beta}^B(P_1,\cdots,P_k) = m_{k,\beta}(P_1,\cdots,P_k)\cdot (Hol_{\partial \beta} \mathcal{L})exp(2\pi \sqrt{-1} \int_{\beta} B)
\end{equation} 
The factor $(Hol_{\partial \beta} \mathcal{L})$ records the holonomy around the boundary of $J$-holomorphic discs.
In the case that $B\neq 0$, the holonomy factor alone is not enough to make it well-defined for each
homotopy class.

It is well-known that the theorems in \cite{FOOO} can be carried over to the
situation with $B$-fields, since the only modification is adding an additional factor to the
boundary operations, which vanishes if the homotopy class $\beta$ is that of constant discs.
 
In particular, the following theorem also holds true with the presence of $B$-fields.
In the case of Lagrangian submanifolds that we consider in this paper, it is weakly obstructed with the bounding cochain $B=0$, hence we omit the bounding cochains from the statement. Also, for simplicity, instead of
$HF((L_1,\CL_1,\nabla_1),\psi_*(L_1,\CL_1,\nabla_1))$, we simply write $HF(L_1,\psi(L_1))$.
\begin{theorem}$($ THEOREM G \cite{FOOO} $)$
$HF(L,\psi(L))$ is independent of $\psi$ and it coincides with  the homology of $\AI$-algebra of $L$.
\end{theorem}
As in \cite{C},\cite{CO}, we will show the non-vanishing of the homology of an $\AI$-algebra (in a non-unitary setting) to prove the non-displaceability of Lagrangian submanifolds.

\subsection{Non-unitary case with a complex valued two form $\widetilde{B}$}
In this subsection, we consider slight modification using a complex valued two form $\widetilde{B}$.
We consider a closed complex valued two form $\widetilde{B} = B_{re} + \sqrt{-1} B_{im}$ such that
the curvature $F_{\nabla_i}$ of the connection $\nabla_i$ on $\CL$ satisfies for $i=1,2$,
$$F_{\nabla_i} = \widetilde{B},$$
instead of (\ref{curv}).
Moreover, we do not require $\widetilde{B}/(2\pi i)$ to lie in the image of $H^2(M;\ZZ) \to H^2(M;\CC)$ as long as there exist 
$(\CL_i,\nabla_i)$ satisfying the above curvature condition for $i=1,2$. (i.e. we only require
integrality condition when restricted to $L_i$, not necessarily on $M$.)

In the case that $B_{re}=0, B_{re}/(2\pi i) \in H^2(M,\ZZ)$, then $\widetilde{B}/(2\pi i)$ becomes a $B$-field discussed in the previous subsection.
 
The contribution from $\widetilde{B}$ is given as
\begin{equation}\label{curv2}
	exp( 2\pi \sqrt{-1} \int_{\beta} B)(Hol_{\partial \beta} \mathcal{L}) \cdot T^{\omega(\beta)} \;\;\; \Rightarrow \;\;\; exp( \int_{\beta} \widetilde{B})(Hol_{\partial \beta} \mathcal{L}) \cdot T^{\omega(\beta)} 
\end{equation}
\begin{lemma}
The new expression in (\ref{curv2}) also depends only on the homotopy class of $\beta$.
\end{lemma}
\begin{proof}
The part of (\ref{curv2}),
$ T^{\omega(\beta)}$  depends only on the homotopy class of $\beta$, since
$ \omega$ are closed two forms which vanish on $TL$. 
The part, $exp( \int_{\beta} \widetilde{B})(Hol_{\partial \beta} \mathcal{L})$ also
 depends only on the homotopy class of $\beta$, due to the relation $F_{\nabla_i} = \widetilde{B}$.
\end{proof}
We remark that in many cases,  $\widetilde{B}$ maybe set to be zero, since considering flat {\it non-unitary} line
bundle is good enough to detect non-displaceability of some Lagrangian submanifolds.
Also, in this paper, we will choose $\widetilde{B}$ so that its support is away from the given Lagrangian submanifolds.
In such a case, integrality condition is automatically satisfied. It is easy to check that if $\widetilde{B}$ satisfies integrality condition when restricted on $L$, then, it also satisfies such a condition on $\psi(L)$ for a Hamiltonian isotopy $\psi$. 

The rest of the story on the definition of Floer homology and their properties are almost identical to the
case of $B$-fields. The benefit of the introducing $\widetilde{B}$ and non-unitary line bundle is that,
as it can be seen in (\ref{curv2}), we now allow the additional contribution to be in $\CC$, whereas
the traditional approach only allows it to be in $S^1 \subset \CC$. 
(We remark that in \cite{FOOO3}, one of their idea is to generalize this further to an element of the Novikov ring based on
our observation)

Extension of Floer cohomology to the non-unitary line bundle does not seem to be
related to the mirror symmetry directly. The reason is that in mirror symmetry,  
one usually identifies the formal parameter $T^{2\pi}$ with a specific value, say $e^{-1}$, and after the identification, the above may be
interpreted as introducing a new symplectic form $\omega + 2\pi B_{re}$ (if it is symplectic):
\begin{equation}\label{curv21}
 e^{\int_{\beta} \widetilde{B}}(Hol_{\partial \beta} \mathcal{L}) \cdot e^{\omega(\beta)/2\pi}
 = e^{  \int_{\beta} (\widetilde{B} + \omega/2\pi)} (Hol_{\partial \beta} \mathcal{L})
\end{equation}
But, in the realm of symplectic geometry, the difference of formal parameter version and
convergent version can be clearly observed (see the last section), hence Floer cohomology with $\widetilde{B}$ in the
Novikov ring coefficient,  is different from the Floer cohomology with a new symplectic form.

\section{Non-displaceable Lagrangian submanifolds}
\subsection{Lagrangian fibers in toric Fano manifolds}
In \cite{C} and \cite{CO}, Floer cohomology of a Lagrangian fiber has been computed 
in a general toric Fano manifold (with convexity assumption, which we remove later in this paper).
The Lagrangian fibers with non-vanishing Floer cohomology in these cases, like the Clifford torus, has
 strong symmetry conditions on the position of the fiber and the shape of the
moment polytope that we start with, hence they were called balanced fibers.
For example, in \cite{CO}, the Hirzbruch surface was shown to have
no Lagrangian fiber with non-vanishing Floer cohomology. 

With the introduction of $\widetilde{B}$, we show that such symmetry condition for
the non-vanishing Floer cohomology, can be weakened considerably, hence providing many more examples of non-displaceable Lagrangian fibers. For example, we prove that the monotone fiber in the Hirzbruch surface is non-displaceable. In fact, we prove that every monotone Lagrangian fiber of toric Fano manifold is non-displaceable.

Let us first recall the notations of toric Fano manifolds from \cite{CO}, to which we refer readers for details.
  For a given polytope $P$, and we consider the corresponding symplectic toric manifold $(M,\omega)$. Let $v_i$ for $i=1,\cdots,N$ be
  the inward normal vectors of each facet of $P$ which define one dimensional generators of the dual fan $\Sigma$.
  Denote by $V(v_i)$ the corresponding toric divisor. Also, recall that the homotopy classes of  holomorphic discs 
  with Maslov index two may be enumerated as $\beta_1,\cdots,\beta_N$ where $\beta_i$ corresponds to
  each facet $v_i$. 

\begin{lemma}\label{Bi}
There exist closed real-valued two forms $B_i$ on $M$ which vanish on $TL$ such that for $i,j = 1,\cdots,N$, we have
 $$\int_{\beta_j} B_i = \delta_{ij}.$$  
\end{lemma}
\begin{proof}
 We have shown that the Maslov index is given by twice the intersection number with toric divisors(Theorem 5.1 \cite{CO}),
 and from the classification theorem(Theorem 5.2 \cite{CO}), each holomorphic discs of class $\beta_i$ (with Maslov index two) intersects non-trivially  only with $V(v_i)$ and it does not intersect any other toric divisors. Hence we may choose Poincare dual closed two form $B_i$ for each
 $V(v_i)$, whose support lies in a sufficiently close neighborhood of $V(v_i)$.
These satisfy the equations above since each Maslov index two holomorphic disc is a part of a holomorphic sphere which does not intersect the corresponding toric divisor elsewhere.
 This proves the lemma.
\end{proof}

We consider the case of monotone Lagrangian fibers first, and we will
discuss general case in the next subsection.
Recall that a Lagragrangian submanifold is called monotone if for any $\beta \in \pi_2(M,L)$,
we have $ \mu(\beta) = c \cdot \omega(\beta)$ for some fixed constant $c \in \RR_{>0}$.
This implies that the symplectic manifold is also monotone, in the sense 
that for any spherical homology class $\alpha \in \pi_2(M)$, $\omega(\alpha) = 2c \cdot c_1(TM)(\alpha)$.

\begin{theorem}\label{thm:mono}
Let $L$ be any monotone Lagrangian torus fiber in any toric Fano manifold $M$.
There exists $\widetilde{B}$ such that the Bott-Morse Floer cohomology of $L$ with $\widetilde{B}$
is non-trivial. Hence $L$ is not displaceable from itself by any Hamiltonian isotopy of $M$.
\end{theorem}
\begin{remark}
Entov and Polterovich proved the same result using the the theory of quasi-state and quasi-measures \cite{EP}.
\end{remark}
\begin{remark}
There are many examples of monotone Lagrangian torus fibers satisfying this fact. 
More explicitly, note that toric Fano manifold may be given a symplectic form so that its moment polytope is reflexive. 
(i.e. its facets are defined by $$<x,v_i> \geq -1,$$ and $0$ being the only integral point in the interior of
the polytope).
Then, Lagrangian fiber corresponding to the origin, satisfies this property since the symplectic
areas of the holomorphic discs are all equal to $2\pi$ from the area formula of \cite{CO}.
Hence, such fibers are always non-displaceable by Hamiltonian isotopy.
\end{remark}
\begin{proof}
As the Maslov index of any non-trivial holomorphic disc is positive, we may write the Bott-Morse differential $m_1$ as
\begin{eqnarray*}
m_1 &=& m_{1,0} + \sum_{\beta \in \pi_2(M,L)} m_{1,\beta} \\
&=& m_{1,0} + m_{1,2} + m_{1,4} + \cdots .
\end{eqnarray*} 
where $m_{1,0}$ is the classical boundary operator of taking the boundary of a chain and
$m_{1,\beta}$ records quantum contributions from $J$-holomorphic
discs of homotopy class $\beta$ intersecting a chain. And by $m_{1,\mu}$, we denote
the sum of $m_{1,\beta}$ over all $\beta$ with the Maslov index $\mu(\beta)= \mu$.
To distinguish with the standard Floer cohomology, we write $m_{1,\mu}^{\widetilde{B}}$ in the
case of Floer cohomology with $\widetilde{B}$.

First we choose the following generators $C_i$ of $H^1(L)$ for $i=1,\cdots,n$.
\begin{definition}\label{generator}
Let $l_i$ be a circle $1\times \cdots S^1 \cdots \times 1$ where
$S^1$ is the $i$-th circle of $(S^1)^n \subset (\CC^*)^n $. Then torus action of $(\CC^*)^n$ on $L$ gives
a corresponding cycle in $L$, which we also denote as $l_i$ by abuse of notation.
For $i=1,\cdots n$, denote by $C_i \in H^1 (L)$ the Poincare dual of the cycle
$$ (-1)^{i-1}(l_1 \times \cdots \times \hat{l_i} \times \cdots l_n).$$
\end{definition}
Now, we show the existence of   good $\widetilde{B}$'s for monotone Lagrangian fibers. 
\begin{lemma}\label{lem}
Let $L$ be a monotone Lagrangian fiber in toric Fano manifold $M$.
We may choose a complex valued two form  $\widetilde{B}$ on $M$ which vanishes on $TL$ such that for all $i =1,\cdots,n$, we have
$$m_{1,2}^{\widetilde{B}}(C_i) = 0.$$
Here we set $\CL$ to be the flat complex line bundle over $L$ with trivial holonomy.
\end{lemma}
Assuming the lemma,  we can apply the following proposition to prove non-displaceability of monotone Lagrangian fibers
(we state the proposition in a more general setup).
\begin{prop}\label{compufloer}
Let $L$ be a positive Lagrangian torus in
a general symplectic manifold (See Definition \ref{positive}).
 Assume that for each $i$, 
$$m_{1,2}^{\widetilde{B}}(C_i) = 0,$$

Then, the Floer cohomology with $\widetilde{B}$ is isomorphic to the singular cohomology as a $\NOV$-module.
 $$HF^{\widetilde{B},*}(L,L) \cong HF^*(L;\NOV).$$
Conversely, if $m_{1,2}^{\widetilde{B}}(C_i) \neq 0$ for some $i$, 
then, Floer cohomology vanishes.
\end{prop}
The proof of this proposition will be given in the next section.
The lemma \ref{lem} and the proposition \ref{compufloer} together prove the theorem \ref{thm:mono}.
\end{proof}

Now we begin  proof of the lemma.
\begin{proof}
The expression $m_{1,2}^{\widetilde{B}}(C_i)$ is defined as
\begin{equation}\label{m12}
m_{1,2}^{\widetilde{B}}(C_i) = 
\sum_{j=1}^N [\CM_1(L,J_0,\beta_j,C_i),ev_0] \cdot exp( \int_{\beta_j} \widetilde{B}) \cdot (Hol_{\partial \beta_j} \mathcal{L}) \cdot T^{\omega(\beta_j)}.	
\end{equation}
Here, $[\CM_1(L,J,\beta_j,C_i),ev_0]$  is a chain obtained as the evaluation image of the moduli space of all $J_0$-holomorphic discs
of class $\beta_j$ which intersects $C_i$. This chain has (expected) dimension $n$, and it is in fact a cycle since
the moduli space of Maslov index two holomorphic discs are closed, which is due to the minimality of the index.
Hence it is a multiple of
fundamental class $[L]$. In fact, it can be exactly computed as (see for \cite{C} for exact signs and details)
$$ (-1)^{n} (+ v_{j1} l_1 + \cdots + v_{jn} l_n) \times 
(-1)^{i-1}(l_1 \times \cdots \times \hat{l_i} \times \cdots l_n)$$
$$ = (-1)^{n} v_{ji} (l_1 \times \cdots \times l_n) = (-1)^n v_{ji} [L]$$
As $\CL$ is a flat bundle with trivial holonomy, hence we may set 
$$Hol_{\partial \beta_j} \CL \equiv 1.$$
Monotonicity implies that the terms from symplectic area of discs $T^{\omega(\beta_j)}$ are independent of $j$, 
and we write  $\omega(\beta_j) = A_0 \in \RR$ for any $j$.
Now, we will determine $\widetilde{B}$ as a linear combination of $B_i$'s of the Lemma \ref{Bi}.
Let 
\begin{equation}
\widetilde{B} = \sum_{j=1}^N c_j B_j, \;\;\; c_j \in \CC.
\end{equation}
Then, we define $d_j \in \CC^*$
$$ d_j:= exp( \int_{\beta_j} \widetilde{B}) = e^{c_j} \in \CC^* $$

Then, the expression (\ref{m12}) is given as
$$m_{1,2}^{\widetilde{B}}(C_i)= \sum_{j=1}^N  (-1)^n v_{ji} [L] d_j T^{A_0} = (-1)^{n} \big(  \sum_{j=1}^N  v_{ji} d_j \big) T^{A_0}.$$
Hence, to prove the lemma, it suffices to choose non-zero $d_j\in \CC^*$ for $j=1,\cdots N$, satisfying  
\begin{equation}\label{vji}
\sum_{j=1}^N  v_{ji} d_j =0, \;\;\textrm{for all} \; i=1,\cdots,n.	
\end{equation}

Recall that $v_i$'s are the one dimensional generators
of the dual fan $\Sigma$ of the polytope $P$, hence we have $N > n$ and $\{v_j|j=1,\cdots, N \}$ span the vector space $\RR^n$.
Therefore the equation (\ref{vji}) clearly has a non-trivial solution by linear algebra.

Now we show that we may obtain a solution of (\ref{vji}) such that  $d_j \neq 0$ for all $j$.
To prove this claim, we show that 
the solution space of (\ref{vji}) is not contained in any coordinate plane of $\RR^N$.
To prove the latter, for a fixed index $i \in \{1,\cdots,n\}$, we find a solution $(d_1,\cdots,d_N)$ of (\ref{vji}) such that $d_i \neq 0$.

We  choose $v_{i_2},\cdots,v_{i_n}$ such that
if we set $v_i = v_{i_1}$, the collection $\{v_{i_1},\cdots,v_{i_n}\}$ defines an $n$-dimensional cone in the fan $\Sigma$. This implies
that $\{v_{i_1},\cdots,v_{i_n}\}$ generate $\RR^n$. We may take another vector $v_{i_{n+1}} \neq v_1$ for some
$1 \leq i_{n+1} \leq N$ such that $v_{i_{n+1}}$ 
does not lie on the subspace generated by $\{v_{i_2},\cdots,v_{i_n}\}$.
This is possible since $\Sigma$ is a complete fan.
Then, it is clear that since $\{v_{i_1},\cdots,v_{i_{n+1}}\}$  is not linearly independent, and we may write
\begin{equation}
	v_{i_{n+1}} = \sum_{j=1}^n d_j v_{i_j}.
\end{equation}
From the construction, we have $d_1 \neq 0$. This proves the claim that the solution space is not contained
in any coordinate plane.
This finishes the proof of the lemma.
\end{proof}

\subsection{General positive Lagrangian submanifolds}
The proof in the monotone case suggests the following generalization in the general symplectic manifold.
In short, we prove that  if there exist
an identity similar to (\ref{vji}) at each symplectic energy level (of holomorphic discs), then it can be shown to be non-displaceable.
 
First, we define 
\begin{definition}\label{positive}
Let $(M,\omega)$ be a symplectic manifold.
A Lagrangian torus $L$ is called positive if there exist a compatible almost complex structure $J$, 
such that
\begin{enumerate}
\item Maslov index of a non-trivial $J$-holomorphic disc is always positive.
\item $J$ is Fredholm regular for any $J$-holomorphic discs of Maslov index two, and
the evaluation map $ev_0:\CM_1(L,\beta) \to L$ is submersive for homotopy classes $\beta$ of
 $J$-holomorphic discs of Maslov index two.
\end{enumerate}
\end{definition}

Let $L$ be a  positive Lagrangian torus in a symplectic manifold $(M,\omega)$ with a compatible
almost complex structure $J$ as in the definition.
If we denote the homotopy classes of $J$-holomorphic discs of Maslov index two as
$\beta_1,\cdots, \beta_N$ and consider
$$\{\omega(\beta_1),\cdots,\omega(\beta_N)\} = \{\lambda_1,\cdots,\lambda_k \}$$
where energy levels, $\lambda_i \in \RR$'s are set to be distinct from each other.
Using the map 
$$\partial : \pi_2(M,L) \to \pi_1(L) \to H_1(L) \cong \ZZ^n,$$
one can define $v_j = \partial \beta_j \in \ZZ^n$ for $j=1,\cdots,N$.

We define the set of indices with energy $\lambda_i$ as
$$I_{\lambda_i} = \{ l \;| \; \omega(\beta_l) = \lambda_i \} \subset \{1,2,\cdots,n\}.$$
We list the elements of $I_{\lambda_i}$ as
$$I_{\lambda_i} = \{ l_{i1},\cdots,l_{ia_i} \}$$

\begin{theorem}\label{generalthm}
Let $(M,\omega)$, $L$, $J$, $\beta_*$, $\lambda_*$ as above. We assume
\begin{enumerate}
\item $L$ be a positive Lagrangian torus.
\item For each energy level $\lambda_i$, there exists $c_{l_{i1}},\cdots,c_{l_{ia_i}} \in \CC^*$ such that	
$$ \sum_{j=1}^{a_i} c_{l_{ij}} v_{l_{ij}} = 0.$$
\item There exist a set of complex-valued two forms on $M$, $\{B_i\}$ for $i=1,\cdots,N$, which
vanishes on $TL$ and satisfies 
\begin{equation}\label{bij}
exp\; (\int_{\beta_j} B_i) = \begin{cases} 1 & \;\;\; \textrm{if}\;\;\; i\neq j \\ c_i & \;\;\;\textrm{if} \;\;\; i=j \end{cases}
\end{equation}
\end{enumerate}
Then, there exist a closed complex valued two form  $\widetilde{B}$ such that Floer cohomology
with $\widetilde{B}$ is non-vanishing, and in fact, isomorphic to 
the singular cohomology $H^*(L,\Lambda_{nov})$ as a module.
\end{theorem}
\begin{remark}
In case $c_i =1$, then $B_i$ can be chosen as zero.
\end{remark}
\begin{proof}
The proof is exactly the same as monotone case, but the difference is that
we apply the same argument for each energy level to achieve the vanishing of $m_{1,2}^{\widetilde{B}}(C_*)$.
\end{proof}
In the last section, we provide some examples of non-monotone Lagrangian torus fibers which is shown to be non-displaceable
as an application.
\section{Computation of Floer cohomology}
This section consists of two parts. In the first part, we give the proof of Proposition \ref{compufloer}, which claims
that in certain cases, we can determine the Floer cohomology of Lagrangian tori from the Maslov index two boundary map computation on the codimension one generators. In the second part, we explain the technical problem of transversality on the computation of Floer cohomology of
Lagrangian tori on the general toric Fano manifolds from \cite{CO}, and explain how to overcome the problem with the idea  of the first part.
We learned that similar idea to that of the first part has been used by Biran-Cornea in a different setting of
Lagrangian Floer cohomology in \cite{BC}.

\subsection{Computation from Maslov index two contribution.}
We begin the proof of the Proposition \ref{compufloer}. In the case of $\widetilde{B}=0$, this will be used to  to remove the convexity assumption in \cite{CO}.

Recall that we have an $\AI$-algebra of Lagrangian submanifolds by the construction of \cite{FOOO}.
In our case, we assume the positivity of Maslov index for non-constant holomorphic discs and
this implies that $m_0$ is a multiple of the fundamental cycle $[L]$. Hence, in the language of
\cite{FOOO}, it is weakly obstructed, and we may set the bounding cochain to be $B=0$. We also
assume that the $\AI$-algebra of \cite{FOOO} is unital.
(We may need to take the canonical model if necessary(\cite{FOOO}) for this. But
in this positive case, the $m_{1,2}^{\widetilde{B}}(C_i)$ part can be made to be the same after
such process due to dimensional restriction)

In Proposition \ref{compufloer}, we have assumed that $m_{1,2}^{\widetilde{B}}(C_i) =0$. (From now on
we omit $\widetilde{B}$ from the notation for simplicity). The above condition actually implies that 
$$m_1(C_i) =0.$$
From the dimension calculation, $m_{1,\mu}(C_i)$ is expected to be $n-2+\mu$ dimensional chain ($n = dim(L)$),
hence for $\mu \geq 4$, $m_{1,\mu}(C_i)$ maybe regarded as zero.

Now, it remains to show that $C_i$'s are {\it not} $m_1$-coboundaries.
(It may become coboundary  because the dimension of $m_{1,\mu}(P)$ equals $n-1 = dim(C_i)$ if  $\mu = deg(P)-1$.)
But we can show that $C_i$'s are not $m_1$ coboundaries by the following simple argument.

By the Leibniz rule from the $\AI$-algebra equations, we have
\begin{equation}\label{leibniz}
m_1(m_2(x,y)) + m_2(m_1(x),y) +(-1)^{|x|+1} m_2(x,m_1(y))=0.  
\end{equation}
If $x,y$ are $m_1$-cycles, then
$m_2(x,y)$ is also a $m_1$-cycle. Recall that its energy zero part is $m_{2,\beta_0}(x,y)$ which is given by
cup product of ( or intersection product of chains)  $x$ and $y$ (here $\beta_0$ is the homotopy class of constant discs). And if $x$ or $y$ is a $m_1$-coboundary, then
$m_2(x,y)$ is also a $m_1$-coboundary. 

For the torus, we have chosen $C_i$'s so that the following intersection is transversal and
equal to
 $$(\cdots(PD(C_1) \cap PD(C_2)) \cap \cdots )\cap PD(C_n)) = \pm pt \in T^n$$
Hence, we may define $m_{2,0}(C_i,C_j)$ so that 
\begin{equation}\label{pteq}
\pm PD(pt)  = m_{2,0}(m_{2,0}(\cdots(m_{2,0}(C_{1},C_{2}), C_{3}),\cdots), C_{n})
\end{equation}
Therefore, we may write 
\begin{equation}\label{ccc}
m_{2}(m_{2}(\cdots(m_{2}(C_{1},C_{2}), C_{3}),\cdots), C_{n}) =  \pm PD(pt) \pm  \sum_{k} a_k P_k T^{\lambda_k},
\end{equation}
where $a_k \in \QQ$, and $\lambda_k >0$ and $P_k$'s are chains of positive dimensions in $L$.
These $P_k$'s are the result of sucessive operations $m_{2,\beta}$'s and easy computation shows that they are of
positive dimensions except (\ref{pteq}). We remark that in \cite{C2} we observed that  $PD(pt)$ is
a singular cohomology cycle, but not a $m_1$-cycle in general. And
one needs additional chains such as $P_k$ to get $m_1$-cycle in the chain level. 

Notice that the expression (\ref{ccc}) is an $m_1$-cycle since we assumed that each $C_i$'s are $m_1$-cycles. 
We claim that
the expression (\ref{ccc}) cannot be a $m_1$-coboundary
since its energy zero part is $PD(pt)$.

Suppose the expression (\ref{ccc}) equals $m_1( \sum_j Q_j T^{\sigma_j})$ for some $Q_j,\sigma_j$'s.
First, note that $m_{1,\beta}(Q_j)$ is of positive dimension if $\mu(\beta)\geq 2$ since
$$dim(Q_j) + \mu(\beta) -1 > 0 .$$
 Hence, to have $PD(pt)$ as a boundary image, 
it should be of the form $m_{1,0}(\sum_{j'} Q_{j'})$ for some $Q_{j'}$'s, but clearly this cannot become a $(pt)$, since
the (signed) sum of the coefficients of the zero dimensional chain  in $m_{1,0}(\sum_{j'} Q_{j'})$ vanishes,
since $m_{1,0}$ is the standard boundary operation $\partial$.

Hence, this proves that (\ref{ccc}) is a non-trivial $m_1$-homology cycle.
Therefore, we conclude that all the intermediate products $m_2(\cdots,m_2(C_1,C_2),\cdots,C_i)$ as well
as all $C_i$'s  are in fact also non-trivial in Floer cohomology.
This proves the main statement of the Proposition.

To prove the converse statement, suppose that $m_{1,2}(C_i) = c_i [L] T^{Area}$ for some
$c_i \neq 0$. This in fact implies that $m_{1,2}(pt) \neq 0$: If we write
$$m_{1,2}(pt) = \sum_{j=1}^n a_j[l_j],$$
then it is not hard to check that $a_i = c_i$ from above.
This implies that $m_{1,2}(pt) \neq 0$.
Now, recall the theorem 10.1 of \cite{CO}, which shows the
equivalence of $m_{1,2}(pt) \neq 0$ and the vanishing of Floer cohomology.
Since the image of Maslov index two is a multiple of fundamental cycle, 
the computation here is very similar to that of \cite{CO} and the only
difference is the fact that in that case, we have used the fact that
$m_{1,\mu} \equiv 0$ for $\mu \geq 4$, which we do not assume here.
But we can work with a spectral sequence comming from the Maslov index (\cite{Oh3}, \cite{FOOO}) after
taking the canonical model,
then the same argument as in \cite{CO} proves the vanishing of Floer cohomology.
Hence the converse statement holds in our case, and this proves the proposition.

\subsection{The case of torus fibers in the general toric Fano manifolds}
We first describe the problem we have encountered in \cite{CO}, and 
we explain how to resolve it. 
Recall that for toric Fano manifolds,
we have classified all holomorphic discs with boundary on Lagrangian torus fibers
generalizing the results of \cite{C}. By proving the Fredholm regularity of the standard
complex structure $J_0$ for all the holomorphic discs, and computing symplectic areas of holomorphic discs, we have
explicitly computed Floer differential $m_{1,\beta}$. The only problem was that although $J_0$ is Fredholm regular for all holomorphic discs,
it is {\it not} Fredholm regular for holomorphic spheres for general toric Fano manifolds.
It means that for the holomorphic spheres, 
the actual dimensions of the moduli spaces may be bigger than the expected (virtual) dimensions of them due to this non-regularity.

Hence, the moduli space of holomorphic discs 
of Maslov index $\geq 4$  
may not have a good compactification since it may bubble off a holomorphic sphere, and in this case
the boundary strata may have bigger dimension than the main stratum.
So we have restrict ourselves to the case of,  so called,  convex symplectic manifolds in \cite{CO}, which
guarantees the regularity of $J_0$-holomorphic spheres.

This non-regularity problem can be resolved if we make (abstract Kuranishi) perturbations of the moduli spaces,
then the perturbed moduli space  is almost impossible to use for computation directly.
Recall that without perturbation, the image of holomorphic discs of Maslov index $\geq 4$ did not have
any non-trivial contribution. More precisely, in Proposition 7.2 \cite{CO}, we have proved that $m_\mu \equiv 0$ for $\mu \geq 4$. This is because 
although the moduli spaces of holomorphic discs have correct dimensions, their evaluation images with
one marked point is of smaller dimension than expected, and it was regarded as zero. In this way, the Floer cohomology  depended only on $m_{1,2}$.
But as soon as we introduce abstract perturbations to resolve the sphere bubble issues, this argument is no longer true  since the images of the virtual chains after abstract perturbations will have images of expected dimension in general.
Hence, their contribution may not vanish. 

Now, the arguments in the last subsection tells us how to overcome this problem, since  Floer cohomology can be determined by $m_{1,2}(C_i)$ only. In our case, since the moduli space of holomorphic discs of Maslov index two
are always well-defined without perturbation, we will use unperturbed moduli spaces for the index two case and
we can make explicit computations on $m_{1,2}(C_i)$. Now, we can introduce an abstract perturbation for $\mu \geq 4$
if necessary. Even though $m_{1,\mu}$ for $\mu \geq 4$ is not computable, the arguments in the previous subsection
can be used to determine Floer cohomology completely. Now it is clear that
the results of \cite{CO} holds for all toric Fano manifolds from the arguments in the first part,
together with the machinery of \cite{FOOO}.
\begin{prop}
The theorems in \cite{CO} also holds true for all toric Fano  manifolds
\end{prop}

Now the interesting open question is what happens in the non-Fano case. The argument itself works
fine even in non-Fano case, but the problem is that the moduli space of holomorphic discs of Maslov index
two also needs to be perturbed, since there might be a sphere bubble with non-positive Chern number from
the Maslov index two disc. This might lead us to interesting questions regarding wall-crossings and
multi-valuedness of $m_0$ as in \cite{A}.

\section{More Examples}
In the first subsection, we consider the  Hirzebruch surface $F_1$ which 
illustrates very well the subtle differences between various versions of Lagrangian Floer cohomology 
considered in this paper. Later, we  discuss non-monotone examples.
\subsection{Hirzebruch surface}
Recall that $F_1$ is obtained from $\CP^2$ by blowing up a fixed point of the torus action.
Symplectic form can be given by specifying its moment polytope, which we define as the region in $\RR^2$ 
(See Figure 1 (a)) bounded by
\begin{equation}
	x=-1,\; y=-1,\; x+y=-1,\; x+y=1.
\end{equation}
First, recall that there are four homotopy classes of holomorphic discs of Maslov index two, 
$\beta_i$ corresponding to each normal vector $v_i$ for $i=1,2,3,4$. Also recall
the following area formula of holomorphic discs from \cite{CO}.
\begin{theorem}(Theorem 8.1 \cite{CO})
Let $P$ be a polytope defining a toric Fano manifold $M$, which is defined as 
$$\{x \in \RR^n| <x,v_j > \geq \lambda_j \}.$$ Let $A = (\theta_1,\cdots,\theta_n) \in \RR^n$.
Then the area of holomorphic disc $D(v_j)$ of homotopy class $\beta_j$ whose boundary lies in the Lagrangian submanifold $\mu_P^{-1}(A)$ is
$$2\pi(<A,v_j> - \lambda_j).$$
\end{theorem}
Hence, one can easily notice that for the $\mu_P^{-1}((0,0))$, the areas of $\beta_1,\beta_2,\beta_3,\beta_4$
all equals $2\pi$ from the formula. The fiber $\mu_P^{-1}((0,0))$ is in fact a monotone Lagrangian submanifold.

Now, one can explicitly compute the Floer cohomology of $\mu_P^{-1}((\theta_1,\theta_2))$  for any 
$(\theta_1,\theta_2) \in P$ as follows.
Let $l_1,l_2$ be two generators of $H_1(T^2)$, and 
without considering $\widetilde{B}$, we have (See \cite{CO} for details) 
\begin{equation} \label{hireq1}
m_1(l_1) = -T^{2\pi(1-\theta_1-\theta_2)} e^{-h_1i-h_2i} + T^{2\pi(1+\theta_1+\theta_2)} e^{h_1i+h_2i} + T^{2\pi(1 + \theta_2)} e^{h_2i}
\end{equation}
\begin{equation}\label{hireq2}
m_1(l_2) = -T^{2\pi(1-\theta_1-\theta_2)} e^{-h_1i-h_2i} + T^{2\pi(1+\theta_1+\theta_2)} e^{h_1i+h_2i} + T^{2\pi(1 + \theta_1)} e^{h_1i}	
\end{equation}
To have non-vanishing Floer cohomology, we should have $m_1(l_1)=m_1(l_2)=0$ from
the Proposition \ref{compufloer}.
Since  two expressions (\ref{hireq1}),(\ref{hireq2}) differ by the last term, one
can easily see that 
\begin{equation}\label{fleq}
\theta_1 = \theta_2=\theta, \;\;  h_1=h_2=h.
\end{equation}
Now, we will discuss three versions of Lagrangian Floer homology separately.
\subsection{The standard Floer cohomology}
As there are three terms to be canceled, it is necessary that all the terms $T^{area}$ should
have equal area, which implies $\theta_1 = \theta_2 = 0$.
Hence we need to solve the equation. 
$$ e^{4hi} + e^{3hi} +1 =0.$$
but this equation does not have a solution for $h \in \RR$ (See (\ref{convnonv})).
Hence, the standard Floer cohomologies of all Lagrangian torus fibers vanish in this case.

\subsection{Convergent version of Floer cohomology}
We first explain the computation in the convergent case and explain its implications afterward.
We recall that in the convergent version of Floer cohomology which appears in Mirror symmetry, 
one substitute a formal parameter $T$ in the Novikov ring with a specific complex value.
In general, the operations $m_*$ is an infinite sum over all quantum contributions where infinite sum 
was justified by  using  the Novikov ring coefficients.
After the substitution,  there is a convergence issue of the infinite sum.
But with the positivity  assumptions, it is only a finite sum from the dimension argument.
Hence in this case there is no convergence problem.

In \cite{CO}, the substitution $T^{2\pi}=e^{-1}$ was used to identify $m_0$ and the Landau-Ginzburg superpotential.
Hence, we will consider the case of the substitution $T^{2\pi}=e^{-1}$. 
Note that in the convergent version, $T$ is  of explicit value, 
hence  the terms with different exponents  could add up.
Then, we set 
$$z=e^{-\theta +ih}.$$
Then, vanishing of (\ref{hireq1}),(\ref{hireq2}), with (\ref{fleq}) is equivalent to
\begin{equation}
z^4 + z^3 -1 =0.
\end{equation}
This equation has four solutions which are approximately (obtained from Matlab).
\begin{equation}\label{convnonv}
 e^{-0.0614 \pm 1.8063 i}, e^{0.1995}, e^{0.3223 + \pi i}.
\end{equation}
Note that none of them corresponds to the origin where $|z| =1$.
But, in \cite{EP}, it was observed that all the fibers other than $\mu_P^{-1}((0,0))$ are
displaceable by some Hamiltonian isotopy. Hence, the fibers corresponding to
(\ref{convnonv}) have non-vanishing Bott-Morse Floer cohomology $HF(L,L)$, but
is displaceable by some Hamiltonian isotopy. This implies that 
for the convergent version of Floer cohomology,
we have $$HF(L,L) \neq HF(L,\psi(L)).$$

This shows that the convergent version does not have the same property as the standard version
at least in the Morse-Bott case. This is in part because the presence of formal parameter $T$ was crucially used in proving various isomorphisms and  the symplectic
invariance property of Floer cohomology. 
But mirror symmetry correspondence seems to be related to the convergent version of 
Floer cohomology.

\subsection{With a complex valued closed two from $\widetilde{B}$}

Now, we show that we may choose a good $\widetilde{B}$ which gives rise to a non-vanishing Floer cohomology
for the monotone Lagrangian fiber $\mu_P^{-1}((0,0))$.
Let 
$$v_1 = (-1,-1), v_2=(1,0), v_3=(1,1), v_4 = (0,1).$$
Then, we can find a solution $(c_1,c_2,c_3,c_4)$ with $c_i \in \CC^*$ satisfying
$$\sum_{i=1}^4 c_i v_i =0.$$
In fact, (\ref{hireq1}),(\ref{hireq2}) becomes (with trivial holonomy for simplicity)
\begin{equation} \label{hireq11}
m_1(l_1) = \big(((-l_1-l_2) \times l_1)e^{\int_{\beta_1} \widetilde{B}}
+ ((l_1 +l_2) \times l_1)e^{\int_{\beta_3} \widetilde{B}}
+ (l_2 \times l_1)e^{\int_{\beta_4} \widetilde{B}} \big) T^{2\pi}
\end{equation}
$$ = \big( e^{\int_{\beta_1} \widetilde{B}} - e^{\int_{\beta_3} \widetilde{B}}- e^{\int_{\beta_4} \widetilde{B}} \big) T^{2\pi} = (c_1 - c_3 - c_4) T^{2\pi}$$
\begin{equation} \label{hireq21}
m_1(l_2) = \big(((-l_1-l_2) \times l_2)e^{\int_{\beta_1} \widetilde{B}}
+ (l_1 \times l_2)e^{\int_{\beta_2} \widetilde{B}} 
+ ((l_1 +l_2) \times l_2)e^{\int_{\beta_3} \widetilde{B}}\big) T^{2\pi}
\end{equation}
$$ = \big( - e^{\int_{\beta_1} \widetilde{B}} + e^{\int_{\beta_2} \widetilde{B}} +e^{\int_{\beta_4} \widetilde{B}} \big) T^{2\pi} = (-c_1 + c_2 + c_3) T^{2\pi}.$$

Such $c_i$'s are easy to find, for example, we may take 
$$c_1 = 2, c_2=c_3=c_4=1.$$
Hence $\widetilde{B}$ may be taken as a suitable constant multiple (so that $c_1=2$) of the Poincare dual of
the divisor $D(v_1)$ whose support is sufficiently close to $D(v_1)$. Then, the resulting
Floer cohomology with $\widetilde{B}$ is non-vanishing, hence proves the non-displaceability of this
monotone fiber. 
In this particular example, another way to prove this is to set $\widetilde{B}=0$ and use non-unitary flat connections to obtain the same results.
\begin{figure}[ht]
\begin{center}
\includegraphics[height=1.5in]{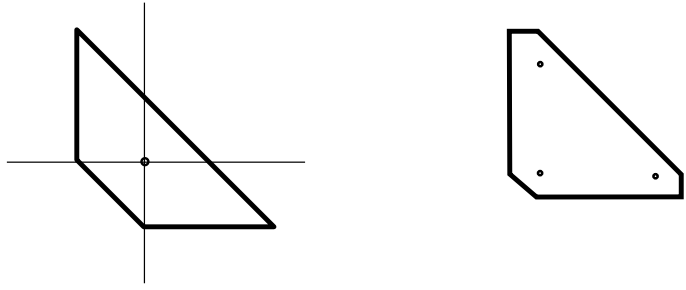}
\caption{(a) Hirzbruch surface, \hspace{1cm} (b) Three point blow up of $\CP^2$ }
\label{fig3}
\end{center}
\end{figure}

\subsection{Non-monotone examples}
We provide some non-monotone Lagrangian submanifolds which are non-displaceable by Hamiltonian isotopy in this subsection, which were not known previously.

The first example is the case of $\CP^2$ blown up at three fixed points of the torus action
(in the same way at each three points) (See Figure 1(b) ). 
In fact, there are three non-displaceable fibers as in the figure. 

We may define the polytope as a  region bounded by the lines
$$ x+ y =1,\; x=0,\; y=0,\;  x= 1-\epsilon,\;  y= 1-\epsilon,\;  x+y = \epsilon.$$
Then, we can check that 
the fibers corresponding to 
$$(\epsilon,\epsilon), \;\; (\epsilon, 1-2\epsilon),\;\; (1-2\epsilon, 1-2\epsilon)$$
are non-displaceable by the theorem \ref{generalthm}.
Let us denote the normal vectors as $v_1,\cdots,v_6$ in this case counting
counter clockwise starting from $v_1 = (-1,-1)$.

One may notice that there are two energy levels, $\lambda_1 = 2\pi \epsilon, \lambda_2= 2\pi(1-2\epsilon)$ of
holomorphic discs of Maslov index two, and
and at the energy level $2\pi\epsilon$, there are three normal vectors
$v_3,v_4,v_5$ with $$v_3 + e^{\pi i} v_4 + v_5 =0,$$ and at the energy level $2\pi(1-2\epsilon)$,
there are three normal vectors $v_1,v_2, v_6$ with
$$ e^{\pi i} v_1 + v_2 + v_6 =0.$$
Hence the non-displaceability follows from the theorem. 

There is a higher dimensional analogue of this example. Namely for $\CP^n$, one can consider
the blow up at $(n+1)$ fixed points of the torus action. Then, there will be $(n+1)$ non-displaceable
fibers as in the above example near each blow up, whose non-displaceability can be proved analogously.

As dimension goes up, it gets easier to find more examples, since there are more normal vectors to play with.
Another example may be obtained by blowing up $\CP^1 \times \CP^1 \times \CP^1$ at the
two fixed points where one is located at the other end of the diagonal from the other.
More precisely, one can set the moment polytope to be the cube with verticies $(\pm 1,\pm 1, \pm 1)$, and the normal vectors corresponding to the blow ups to be $(1,1,1),(-1,-1,-1)$. Then, if the corresponding facets are
given by $x+y+z =1$ and $x+y+z=-1$, then one gets a reflexive polytope.
Now, to get non-monotone examples, one moves four facets (out of eight) together. Namely
\begin{enumerate}
\item Consider $x+y+z =1-\epsilon$, together with $x=y=z=-1+\epsilon$ for some small $\epsilon >0$ and
four other facets remain to be the same.
Then, the fiber corresponding to $(0,0,0)$ is not monotone, we can see that there are four normal vectors 
at the energy level $2\pi$ and $2\pi(1-\epsilon)$ each. One can show that the fiber corresponding to
$(0,0,0)$ is non-displaceable, in fact by the standard Floer cohomology, from the cancellation argument.  
\item Consider the case of $x+y+z =1-\epsilon$, $x=y=z=1-\epsilon$. In this case
the fiber corresponding to $(0,0,0)$ is non-displaceable, using the Floer cohomology with $\widetilde{B}$
since cancellation arguments work with the help of $e^{\pi i }$ as in the previous example.
\end{enumerate}

\bibliographystyle{amsalpha}

\begin{thebibliography}{FOOO}
\bibitem[A]{A} D. Auroux {\em Mirror symmetry and T-duality in the complement of an anticanonical divisor} preprint,
arXiv:0706.3207
\bibitem[BC]{BC} P. Biran, O. Cornea, {\em Quantum structures for Lagrangian submanifolds,} preprint, 
arXiv:0708.4221
\bibitem[BEP]{BEP} P. Biran, M. Entov, L. Polterovich 
{\em Calabi quasimorphisms for the symplectic ball.} Commun. Contemp. Math. 6 (2004), no. 5, 793-802. 
\bibitem[C]{C} C.-H. Cho,
{\em Holomorphic discs, spin structures and the Floer cohomology
of the Clifford torus,}  Int. Math. Res. Not.  (2004),  no. 35, 1803-1843.
\bibitem[C2]{C2} C.-H. Cho,
{\em Product structures of the Floer cohomology of Lagrangian torus fibers
in toric Fano manifolds}  Comm. Math. Phys. 260, (2006). 613-640    
\bibitem[C3]{C3} C.-H. Cho,
{\em Non-displaceable Lagrangian submanifolds and Floer cohomology with non-unitary line bundle}  Preprint, arXiv:0710.5454
\bibitem[CO]{CO} C.-H. Cho, Y.-G. Oh
{\em Floer cohomology and disc instantons of Lagrangian torus fibers
in toric Fano manifolds}
Asian Journ. Math. 10 (2006), 773-814
\bibitem[CS]{CS} Y. Chekanov, F. Schlenk, in preparation
\bibitem[ElP]{ElP} Y. Eliashberg, L. Polterovich, {\em The problem of Lagrangian knots in four manifolds,}
Geometric Topology (Athens, 1993), AMS/IP stud. Adv. Math., Amer. Math. Soc., (1997), 313-327
\bibitem[EnP]{EP} M. Entov, L. Polterovich 
{\em Quasi-states and symplectic intersections} Comment. Math. Helv. 81 (2006), no. 1, 75-99
\bibitem[EnP2]{EP2} M. Entov, L. Polterovich 
{\em Rigid subsets of symplectic manifolds} preprint, arXiv:0704.0105
\bibitem[Fl]{Fl} A. Floer, {\em Morse theory for Lagrangian
intersections,} J. Differ. Geom. 28 (1988), 513-547.
\bibitem[Fu]{Fu} K. Fukaya,
{\em Floer homology for families-- a progress report.}
Integrable systems, topology, and physics (Tokyo, 2000), Contemp. Math., 309, Amer. Math. Soc. (2002) 33--68
\bibitem[FOOO1]{FOOO} K. Fukaya, Y.-G. Oh, H. Ohta and K. Ono,
{\em Lagrangian intersection Floer theory-anomaly and
obstruction,} Kyoto University preprint, (2000), revised version, (2007).
\bibitem[FOOO2]{FOOO3} K. Fukaya, Y.-G. Oh, H. Ohta and K. Ono,
{\em Lagrangian Floer theory of toric Fano manifolds I}
Preprint, arXiv:0802.1703
\bibitem[HV]{HV}
K. Hori and C. Vafa, {\em Mirror symmetry, } preprint, 2000,
hep-th/0002222.

\bibitem[K]{K} M. Kontsevich, {\em Homological algebra of mirror
symmetry,} ICM-1994 Proceedings, Z\"urich, Birkh\"auser, (1995).
\bibitem[Oh1]{Oh1} Y.-G. Oh,
{\em Floer cohomology of Lagrangian intersections and
pseudo-holomorphic discs I,} Comm. Pure and Appl. Math. 46 (1993),
949-994 addenda, ibid, 48 (1995), 1299-1302.
\bibitem[Oh2]{Oh2} Y. -G. Oh,
{\em Floer cohomology of Lagrangian intersections and
pseudo-holomorphic discs II \em{\&} III,} Comm. Pure Appl. Math.
46 (1993), 995-1012 \& The Floer Memorial Volume (ed. by H. Hofer
et al.), Birkh\"auser, Basel, pp 555-573, 1995.
\bibitem[Oh3]{Oh3} Y.-G. Oh,
{\em Floer cohomology, spectral sequence, and the Maslov class of
Lagrangian embeddings,} IMRN (1996), No.7., 305-346.
\end{thebibliography}

\end{document}